\documentclass[10pt,twoside,a4paper]{amsart}

\usepackage[export]{adjustbox}
\usepackage{graphicx,amsthm,amsmath,geometry,hyperref}
\usepackage{epstopdf}
\usepackage{subcaption}
\usepackage[usenames]{color}
\usepackage{amssymb}
\usepackage{amscd}

\usepackage{amsfonts}

\usepackage{enumitem}
\usepackage{stmaryrd}

\newtheorem{thm}{Theorem}[section]

\newtheorem{lemma}[thm]{Lemma}

\newtheorem{rem}[thm]{Remark}
\newtheorem{conj}[thm]{Conjecture}

\numberwithin{equation}{section}

\newcommand{\s}{\sigma}
\newcommand{\nnn}{n=p_1^{a_1}.p_2^{a_2}.p_3^{a_3}.p_4^{a_4}}
\newcommand{\ord}{\textup{ord}}
\newcommand{\ff}{f(a_1,a_2,a_3,a_4)}
\newcommand{\nn}{p_1^{a_1}.p_2^{a_2}.p_3^{a_3}.p_4^{a_4}}
\newcommand{\df}{p_1^{b_1}.p_2^{b_2}.p_3^{b_3}.p_4^{b_4}}
\newcommand{\dff}{p_1^{a_1-b_1}.p_2^{a_2-b_2}.p_3^{a_3-b_3}.p_4^{a_4-b_4}}
\newcommand{\aimod}{a_i\equiv 0 \pmod 2, i=1,2,3,4}
\newcommand{\g}{g(a_1,a_2,a_3,a_4)}

\theoremstyle{remark}

\numberwithin{equation}{section}

\makeatletter
\edef\orig@output{\the\output}
\output{\setbox\@cclv\vbox{\unvbox\@cclv\vspace{0pt plus 20pt}}\orig@output}
\makeatother

\begin{document}
	
\title[Deficient Perfect Numbers]{On Deficient Perfect Numbers with Four Distinct Prime Factors}

\author[P. Dutta]{Parama Dutta}
\address{Department of Mathematical Sciences, Tezpur University, Napaam 784028, Dist. Sonitpur, Assam, India}
\email{parama@gonitsora.com}

\author[M. P. Saikia]{Manjil P. Saikia}
\address{Fakult\"at f\"ur Mathematik, Universit\"at Wien, Oskar-Morgensten-Platz 1, 1090 Vienna, Austria}
\email{manjil.saikia@univie.ac.at, manjil@gonitsora.com}

\date{}

\begin{abstract}
For a positive integer $n$, if $\sigma(n)$ denotes the sum of the positive divisors of $n$, then $n$ is called a deficient perfect number if $\sigma(n)=2n-d$ for some positive divisor $d$ of $n$. In this paper, we prove some results about odd deficient perfect numbers with four distinct prime factors.
\end{abstract}

\subjclass[2010]{Primary 11A25; Secondary 11A41, 11B99.}

\keywords{almost perfect numbers, deficient perfect numbers, near perfect numbers.}

\maketitle

\section{Introduction}\label{intro}

For a positive integer $n$, the functions $\s(n)$ and $\omega(n)$ denote the sum and number of distinct positive prime divisors of $n$ respectively. Such an $n$ is called a perfect number if $\s(n)=2n$. These type of numbers have been studied since antiquity and several generalizations of these numbers have appeared over the years (see \cite{mps} and the references therein for some of them). In fact, one of the most outstanding problems in number theory at the moment is to determine whether an odd perfect number exists or not.

Let $d$ be a proper divisor of $n$. We call $n$ a near perfect number with redundant divisor $d$ if $\s(n)=2n+d$; and a deficient perfect number with deficient divisor $d$ if $\s(n)=2n-d$. If $d=1$, then such a deficient perfect number is called an almost perfect number. Several results have been proved about these classes of numbers: for instance, Kishore \cite{kk} proved that if $n$ is an odd almost perfect number then $\omega(n)\geq 6$, Pollack and Shevelev \cite{ps} found upper bounds on the number of near perfect numbers and characterized three different types of such numbers for even values, Ren and Chen \cite{conj} found all near perfect numbers with two distinct prime factors, Tang, and Ren and LI \cite{collo} showed that no odd near perfect number exists with three distinct prime factors and determined all deficient perfect numbers with two distinct prime factors. In a similar vein, Tang and Feng \cite{df} showed that no odd deficient perfect number exists with three distinct prime factors. Recently, Tang, Ma and Feng \cite{colloq} showed that there exists only one odd near perfect number with four distinct prime divisors. The smallest known odd deficient perfect number with four distinct prime factors is $9018009 = 3^2.7^2.11^2.13^2$, and it is the only such number until $2.10^{12}$.

In this paper, we extend the work of Tang and Feng \cite{df} and prove the following main result.

\begin{thm}\label{th-1}
If $n$ is an odd deficient perfect number with four distinct prime factors $p_1, p_2, p_3$ and $p_4$ such that $n=p_1^{a_1}.p_2^{a_2}.p_3^{a_3}.p_4^{a_4}$ with $p_1<p_2<p_3<p_4$ and $a_1,a_2,a_3,a_4 \geq 1$, then

\begin{enumerate}
\item $p_1=3$, and
\item $5\leq p_2 \leq 7$.
\end{enumerate}
\end{thm}

This paper is organized as follows: in Section \ref{proof} we state and prove several lemmas which will be used in proving Theorem \ref{th-1}; finally in Section \ref{open} we state other results that can be obtained by our methods and state a few conjectures.

\section{Proof of Theorem \ref{th-1}}\label{proof}

We shall prove Theorem \ref{th-1} as a series of lemmas in this section. Before, we state our results, we note the following result from Tang and Feng \cite{df}.

\begin{lemma}[Lemma 2.1, \cite{df}]\label{lem-df}
Let $n=\prod_{i=1}^kp_i^{a_i}$ be the canonical prime factorization of $n$. If $n$ is an odd deficient perfect number, then all the $a_i$'s are even for all $i$.
\end{lemma}

Before we proceed with our results, let us fix a few notations. Throughout this paper, unless otherwise mentioned we take $\nnn$ with $p_1<p_2<p_3<p_4$ distinct odd primes and $a_i$'s to be natural numbers. In light of Lemma \ref{lem-df} all the $a_i$'s are even. If $a$ is any integer relatively prime to $m$ such that $k$ is the smallest positive integer for which $a^k\equiv 1 \pmod m$ then, we say that $k$ is the order of $a$ modulo $m$ and denote it by $\ord_m(a)$. We also define the following function which we shall use very often in this paper $$\ff=\left(1-\frac{1}{p_1^{a_1+1}}\right)\left(1-\frac{1}{p_2^{a_2+1}}\right)\left(1-\frac{1}{p_3^{a_3+1}}\right)\left(1-\frac{1}{p_4^{a_4+1}}\right).$$ Most of the time, we shall skip specifying the $p_i$'s and the $a_i$'s if they are evident from the context.

Assuming that $n$ is an odd deficient perfect number with deficient divisor $d=\df$, then we have 
\begin{equation}\label{eq-main}
    \s(\nn)=2.\nn-\df,
\end{equation}

\noindent where $b_i\leq a_i$. Also write $D=\dff$. Then we have 
\begin{equation}\label{eq-3.2}
    2=\frac{\s(n)}{n}+\frac{d}{n}=\frac{\s(n)}{n}+\frac{1}{D}.
\end{equation}

\noindent An inequality which we will use without commentary in the following is $$\frac{\s(n)}{n}<\frac{p_1.p_2.p_3.p_4}{(p_1-1).(p_2-1).(p_3-1).(p_4-1)}.$$

\begin{lemma}\label{p1}
If $n$ is an odd deficient perfect number of the form in Theorem \ref{th-1}, then $p_1=3$.
\end{lemma}

\begin{proof}
If $p_1\geq 5$, then from equation \ref{eq-3.2} we have $$2=\frac{\s(n)}{n}+\frac{d}{n}<\frac{5.7.11.13}{4.6.10.12}+\frac{1}{5}<2,$$ which is impossible. So, $p_1=3$.
\end{proof}

\begin{lemma}\label{p2}
If $n$ is an odd deficient perfect number of the form in Theorem \ref{th-1}, then $p_2\leq 23$.
\end{lemma}

\begin{proof}
If $p_2\geq 29$, then we have $$2=\frac{\s(n)}{n}+\frac{d}{n}<\frac{3.29.31.37}{2.28.30.36}+\frac{1}{3}<2,$$ which is impossible. Hence $p_2\leq 23$.

\end{proof}

We shall now, look at various cases for $p_2$ in the following series of lemmas. The techniques are always similar, so for the sake of brevity we omit few details, but we will always specify how we can check them.

\begin{lemma}\label{p2-1}
If $n$ is an odd deficient perfect number of the form in Theorem \ref{th-1}, then $p_2\neq 23$.
\end{lemma}

\begin{proof}
If $p_2=23$, then using similar methods like before, we can conclude that $p_3\leq 31$. This gives us two choices for $p_3$, namely $29,31$. We shall look into them separately.

\textit{Case 1.} $p_3=29$.

In this case, we can conclude that $p_4\leq 47$ using similar techniques.

Let $D\geq 9$, then we have $$2=\frac{\s(n)}{n}+\frac{1}{D}<\frac{3.23.29.31}{2.22.28.30}+\frac{1}{9}<2,$$ which is impossible. So, $D=3$ in this case, which means $a_1-b_1=1$ and $a_i=b_i$, $i=2,3,4$. Thus, 

\begin{equation}\label{ord-5}
    \s(3^{a_1}.23^{a_2}.29^{a_3}.p_4^{a_4})=5.3^{a_1-1}.23^{a_2}.29^{a_3}.p_4^{a_4}.
\end{equation}

\noindent We note that $\ord_5(3)=\ord_5(23)=\ord_5(37)=\ord_5(43)=\ord_5(47)=4$, $\ord_5(29)=2$ are all even; but $a_i\equiv 0 \pmod 2$, $i=1,2,3,4$, which means that $5$ does not divide the left hand side of equation \eqref{ord-5}, and this is a contradiction. Further if $p_4=31$, then $\ord_{31}(3)=30,\ord_{31}(23)=10,\ord_{31}(29)=10$ are all even and $\aimod$, so equation \eqref{ord-5} cannot hold. Again, if $p_4=41$, then $\ord_{41}(3)=8, \ord_{41}(23)=10, \ord_{41}(29)=40$ are all even and again equation \eqref{ord-5} cannot hold.

Hence, $p_3\neq 29$.

\textit{Case 2.} $p_3=31$.

Let $p_4\geq 37$, then we have $$2=\frac{\s(n)}{n}+\frac{1}{D}<\frac{3.23.31.37}{2.22.30.36}+\frac{1}{3}<2,$$ which is impossible. So, this case is not possible.

Combining the two cases together, we conclude that $p_2\neq 23$

\end{proof}

\begin{lemma}\label{p2-2}
If $n$ is an odd deficient perfect number of the form in Theorem \ref{th-1}, then $p_2\neq 19$.
\end{lemma}

\begin{proof}

If $p_2=19$, then like before we can conclude that $p_3\leq 37$. This gives us the choices $23,29,31$ and $37$ for $p_3$. Using the elementary inequality $\dfrac{p}{p-1}>\dfrac{p+l}{p+l-1}$ for positive integers $p$ and $l$ we see that $D\geq 9$ cannot occur in this case, if $D\geq 9$ cannot occur when $p_3=23$. And indeed this is the case, since $$2=\frac{\s(n)}{n}+\frac{1}{D}<\frac{3.19.23.29}{2.18.22.28}+\frac{1}{9}<2,$$ is impossible. So, $D=3$ in all these cases, and analogous to equation \eqref{ord-5} we have the following 

\begin{equation}\label{ord-5-2}
        \s(3^{a_1}.19^{a_2}.p_3^{a_3}.p_4^{a_4})=5.3^{a_1-1}.19^{a_2}.p_3^{a_3}.p_4^{a_4}.
\end{equation}

Let us use the function $f$ defined earlier; which is this case is $$\ff=\left(1-\frac{1}{3^{a_1+1}}\right)\left(1-\frac{1}{19^{a_2+1}}\right)\left(1-\frac{1}{p_3^{a_3+1}}\right)\left(1-\frac{1}{p_4^{a_4+1}}\right).$$ We also introduce the following function $$\g=\frac{2^2.5.(p_3-1).(p_4-1)}{19.p_3.p_4}.$$

\noindent From equation \eqref{ord-5-2}, it is clear that in this case $$\ff=\g.$$

If $a_1=2$, then $13$ divides the left hand side of \eqref{ord-5-2}, but it does not divide the right hand side of equation \eqref{ord-5-2}, so this is a contradiction. Similarly, if $a_1=4$, then $11$ divides the left hand side of \eqref{ord-5-2}, but it does not divide the right hand side of equation \eqref{ord-5-2}, so this is a contradiction. So $a_1\geq 6$.

\textit{Case 1.} $23 \leq p_2 \leq 37$.

We have here, 

\begin{align*}
 \ff\geq \left(1-\frac{1}{3^{7}}\right)\left(1-\frac{1}{19^3}\right)\left(1-\frac{1}{23^{3}}\right)\left(1-\frac{1}{29^{3}}\right)&\\
 =0.999274\cdots ,
\end{align*}

\noindent and $$\g\leq \frac{2^2.5.36.40}{19.37.41}=0.999202\cdots .$$ Clearly, this is not possible.

Since, we want to check only inequalities of the form $\ff \geq Q$ and $\g\leq R$ and then compare the values of $Q$ and $R$, so we need to only verify for the smallest possible values of $p_i$'s for $\ff$ and the largest possible values of $p_i$'s for $\g$. So, the above verification need not be done for all sets of possible values of $p_i$'s. This observation will be used later without commentary.

\textit{Case 2.} $p_3=41$.

If $p_4\geq 43$, then we have $$2=\frac{\s(n)}{n}+\frac{d}{n}<\frac{3.19.41.43}{2.18.40.42}+\frac{1}{3}<2,$$ which is not possible.

\textit{Case 3.} $p_3=43$.

If $p_4\geq 47$, then we have $$2=\frac{\s(n)}{n}+\frac{d}{n}<\frac{3.19.43.47}{2.18.42.46}+\frac{1}{3}<2,$$ which is not possible.

\end{proof}

The proof of the following is very similar to Lemma \ref{p2-2}.

\begin{lemma}\label{p2-3}
If $n$ is an odd deficient perfect number of the form in Theorem \ref{th-1}, then $p_2\neq 17$.
\end{lemma}

\begin{proof}

If $p_2=17$, then like before we have $p_3\leq 47$, so the choices of $p_3$ are $19,23,29,31,37,41,43$ and $47$. Nothing again the elementary inequality $\dfrac{p}{p-1}>\dfrac{p+l}{p+l-1}$ for positive integers $p$ and $l$ we see that $D\geq 9$ cannot occur in this case, if $D\geq 9$ cannot occur when $p_3=19$. And indeed this is the case, since $$2=\frac{\s(n)}{n}+\frac{1}{D}<\frac{3.17.19.23}{2.16.18.22}+\frac{1}{9}<2,$$ is impossible. So, $D=3$ in all these cases, and analogous to equation \eqref{ord-5} we have the following 

\begin{equation}\label{ord-5-3}
        \s(3^{a_1}.17^{a_2}.p_3^{a_3}.p_4^{a_4})=5.3^{a_1-1}.17^{a_2}.p_3^{a_3}.p_4^{a_4}.
\end{equation}

Let us use the function $f$ defined earlier; which is this case is $$\ff=\left(1-\frac{1}{3^{a_1+1}}\right)\left(1-\frac{1}{17^{a_2+1}}\right)\left(1-\frac{1}{p_3^{a_3+1}}\right)\left(1-\frac{1}{p_4^{a_4+1}}\right).$$ We also introduce the following function $$\g=\frac{2^5.5.(p_3-1).(p_4-1)}{3^2.17.p_3.p_4}.$$

\noindent From equation \eqref{ord-5-3}, it is clear that in this case $$\ff=\g.$$

If $a_1=2$, then $13$ divides the left hand side of \eqref{ord-5-3}, but it does not divide the right hand side of equation \eqref{ord-5-3}, so this is a contradiction. Similarly, if $a_1=4$, $11$ divides the left hand side of \eqref{ord-5-3}, but it does not divide the right hand side of equation \eqref{ord-5-3}, so this is a contradiction. So $a_1\geq 6$.

\textit{Case 1.} $19\leq p_3\leq 41$.

For this case, we have 

\begin{align*}
 \ff\geq \left(1-\frac{1}{3^{7}}\right)\left(1-\frac{1}{17^3}\right)\left(1-\frac{1}{19^{3}}\right)\left(1-\frac{1}{23^{3}}\right)&\\
 =0.999111\cdots ,
\end{align*}

\noindent and $$\g\leq \frac{2^5.5.40.42}{3^2.17.41.43}=0.996519\cdots .$$ Clearly, this is not possible, so this case cannot occur.

\textit{Case 2.} $p_3=43$.

If $p_4\geq 53$, then we have $$2=\frac{\s(n)}{n}+\frac{d}{n}<\frac{3.17.43.53}{2.16.42.52}+\frac{1}{3}<2,$$ which is not possible. So, $p_4=47$. However, we have $\ord_5(3)=\ord_5(17)=\ord_5(43)=\ord_5(47)=4$, hence $5$ cannot divide the left hand side of equation \eqref{ord-5-3}. Hence, this case is not possible.

\textit{Case 2.} $p_3=47$.

If $p_4\geq 53$, then we have $$2=\frac{\s(n)}{n}+\frac{d}{n}<\frac{3.17.47.53}{2.16.46.52}+\frac{1}{3}<2,$$ which is not possible. So, this case is impossible.

Combining the two cases above, we have $p_2\neq 17$.

\end{proof}

\begin{lemma}\label{p2-4}
If $n$ is an odd deficient perfect number of the form in Theorem \ref{th-1}, then $p_2\neq 13$.
\end{lemma}

\begin{proof}
If $p_2=13$ and $p_3\geq 79$ then we have $$2=\frac{\s(n)}{n}+\frac{d}{n}<\frac{3.13.79.83}{2.12.78.82}+\frac{1}{3}<2,$$ which is impossible. So, $p_3\leq 73$. 

Therefore, the choices of $p_3$ lies in the set $$\{17, 19, 23, 29, 31, 37, 41, 43, 47, 53, 59, 61, 67, 71,73\}.$$ Nothing again the elementary inequality $\dfrac{p}{p-1}>\dfrac{p+l}{p+l-1}$ for positive integers $p$ and $l$ we see that $D\geq 9$ cannot occur in this case, if $D\geq 9$ cannot occur when $p_3=17$. And indeed this is the case, since $$2=\frac{\s(n)}{n}+\frac{1}{D}<\frac{3.13.17.19}{2.12.18.18}+\frac{1}{9}<2,$$ is impossible. So, $D=3$ in all these cases, and analogous to equation \eqref{ord-5} we have the following 

\begin{equation}\label{ord-5-4}
        \s(3^{a_1}.13^{a_2}.p_3^{a_3}.p_4^{a_4})=5.3^{a_1-1}.13^{a_2}.p_3^{a_3}.p_4^{a_4}.
\end{equation}

Let us use the function $f$ defined earlier; which is this case is $$\ff=\left(1-\frac{1}{3^{a_1+1}}\right)\left(1-\frac{1}{13^{a_2+1}}\right)\left(1-\frac{1}{p_3^{a_3+1}}\right)\left(1-\frac{1}{p_4^{a_4+1}}\right).$$ We also introduce the following function $$\g=\frac{2^3.5.(p_3-1).(p_4-1)}{3.13.p_3.p_4}.$$

\noindent From equation \eqref{ord-5-4}, it is clear that in this case $$\ff=\g.$$

\textit{Case 1.} $17\leq p_3\leq 29$.

We note that

\begin{align*}
 \ff\geq \left(1-\frac{1}{3^{3}}\right)\left(1-\frac{1}{13^{3}}\right)\left(1-\frac{1}{17^{3}}\right)\left(1-\frac{1}{19^{3}}\right)&\\
 =0.962188\cdots ,
\end{align*}
\noindent and  $$\g \leq \frac{2^3.5.28.30}{3.13.29.31}=0.95833\cdots,$$ which is not possible. So, this case is not possible.

\textit{Case 2.} $p_3\geq 31$

If $a_1=2$ and $p_3\geq 31$, then we have $$2=\frac{\s(n)}{n}+\frac{d}{n}<\frac{\s(3^2).13.31.37}{3^2.12.30.36}+\frac{1}{3}<2,$$ which is not possible.

If $a_1=4$, then $11$ divides the left hand side of equation \eqref{ord-5-4}, but not the right hand side. So, this is not possible. 

Let $a_1\geq 6$. Then we have 

\begin{align*}
 \ff\geq \left(1-\frac{1}{3^{7}}\right)\left(1-\frac{1}{13^{3}}\right)\left(1-\frac{1}{31^3}\right)\left(1-\frac{1}{37^3}\right)&\\
 =0.999035 \cdots ,
\end{align*}

\noindent and  $$\g \leq \frac{2^3.5.72.78}{3.13.73.79}=0.998786\cdots ,$$ which is not possible. So, this case is not possible.

Combining the two cases above, we conclude that $p_2\neq 13$.

\end{proof}

\begin{lemma}\label{p2-5}
If $n$ is an odd deficient perfect number of the form in Theorem \ref{th-1}, then $p_2\neq 11$.
\end{lemma}

\begin{proof}
If $p_2=11$ and $p_3\geq 199$ then we have $$2=\frac{\s(n)}{n}+\frac{d}{n}<\frac{3.11.199.211}{2.10.198.210}+\frac{1}{3}<2,$$ which is impossible. So, $p_3\leq 197$.

\textit{Case 1.} $p_3\geq 17$.

Nothing again the elementary inequality $\dfrac{p}{p-1}>\dfrac{p+l}{p+l-1}$ for positive integers $p$ and $l$ we see that $D\geq 9$ cannot occur in this case, if $D\geq 9$ cannot occur when $p_3=17$. And indeed this is the case, since $$2=\frac{\s(n)}{n}+\frac{1}{D}<\frac{3.11.17.19}{2.10.18.18}+\frac{1}{9}<2,$$ is impossible. So, $D=3$ in all these cases, and analogous to equation \eqref{ord-5} we have the following 

\begin{equation}\label{ord-5-5}
        \s(3^{a_1}.11^{a_2}.p_3^{a_3}.p_4^{a_4})=5.3^{a_1-1}.11^{a_2}.p_3^{a_3}.p_4^{a_4}.
\end{equation}

We note here that, if $a_1=2$, then $13$ divides the left hand side of equation \eqref{ord-5-5}, but not the right hand side. So, $a_1\geq 4$.

\textit{Subcase 1.1.} $p_3\leq 127$.

Let us use the function $f$ defined earlier; which is this case is $$\ff=\left(1-\frac{1}{3^{a_1+1}}\right)\left(1-\frac{1}{11^{a_2+1}}\right)\left(1-\frac{1}{p_3^{a_3+1}}\right)\left(1-\frac{1}{p_4^{a_4+1}}\right).$$ We also introduce the following function $$\g=\frac{2^2.5^2.(p_3-1).(p_4-1)}{3^2.11.p_3.p_4}.$$

\noindent From equation \eqref{ord-5-5}, it is clear that in this case $$\ff=\g.$$

If $p_3\leq 127$, then we have 

\begin{align*}
 \ff\geq \left(1-\frac{1}{3^{5}}\right)\left(1-\frac{1}{11^{3}}\right)\left(1-\frac{1}{17^{3}}\right)\left(1-\frac{1}{19^{3}}\right)&\\
 =994789\cdots ,
\end{align*}

\noindent and $$\g\leq \frac{2^2.5^2.126.130}{3^2.11.127.131}=0.994497\cdots ;$$ which are incompatible with each other. Hence, this subcase cannot occur.

\textit{Subcase 1.2.} $127\leq p_3\leq 137$.

If $a_2=2$, then we find that $19$ divides the left hand side of equation \eqref{ord-5-5}, but not the right hand side. So, $a_2\geq 4$.

We have 

\begin{align*}
 \ff\geq \left(1-\frac{1}{3^{5}}\right)\left(1-\frac{1}{11^{5}}\right)\left(1-\frac{1}{127^{3}}\right)\left(1-\frac{1}{131^{3}}\right)&\\
 =0.995878\cdots ,
\end{align*}

\noindent and $$\g\leq \frac{2^2.5^2.136.138}{3^2.11.137.139}=0.995514\cdots ;$$ which are incompatible with each other. Hence, this case cannot occur.

\textit{Subcase 1.3.} $139\leq p_3\leq 181$.

If $a_1=4$, then we have $$n=\frac{\s(n)}{n}+\frac{d}{n}<\frac{\s(3^4).11.139.149}{3^4.10.138.148}+\frac{1}{3}<2,$$ which is not possible. So, $a_1\geq 6$ in this case. If $a_2=2$, then we find that $19$ divides the left hand side of equation \eqref{ord-5-5}, but not the right hand side. So, $a_2\geq 4$.

We have 

\begin{align*}
 \ff\geq \left(1-\frac{1}{3^{7}}\right)\left(1-\frac{1}{11^{5}}\right)\left(1-\frac{1}{139^{3}}\right)\left(1-\frac{1}{149^{3}}\right)&\\
 =0.999536\cdots ,
\end{align*}

\noindent and $$\g\leq \frac{2^2.5^2.180.190}{3^2.11.181.191}=0.999261\cdots ;$$ which are incompatible with each other. Hence, this case cannot occur.

\textit{Subcase 1.4.} $p_3=191$ or $193$.

If $p_4\geq 211$ then we have $$n=\frac{\s(n)}{n}+\frac{d}{n}<\frac{3.11.191.211}{2.10.190.210}+\frac{1}{3}<2,$$ which is not possible. So, $p_4\leq 197$.

If $a_1=4$, then we have $$n=\frac{\s(n)}{n}+\frac{d}{n}<\frac{\s(3^4).11.191.193}{3^4.10.190.192}+\frac{1}{3}<2,$$ which is not possible. So, $a_1\geq 6$ in this case.

If $a_1=6$, then $1093$ divides both sides of equation \eqref{ord-5-5}, which means $p_4=1093$, which is impossible. So, $a_1\geq 8$.

If $a_2=2$, then we find that $19$ divides the left hand side of equation \eqref{ord-5-5}, but not the right hand side. So, $a_2\geq 4$.

We have 

\begin{align*}
 \ff\geq \left(1-\frac{1}{3^{9}}\right)\left(1-\frac{1}{11^{5}}\right)\left(1-\frac{1}{191^{3}}\right)\left(1-\frac{1}{193^{3}}\right)&\\
 =0.999943\cdots ,
\end{align*}

\noindent and $$\g\leq \frac{2^2.5^2.192.196}{3^2.11.193.197}=0.999766\cdots ;$$ which are incompatible with each other. Hence, this case cannot occur.

\textit{Subcase 1.5.} $p_3=197.$

If $p_4\geq 211$ then we have $$n=\frac{\s(n)}{n}+\frac{d}{n}<\frac{3.11.197.211}{2.10.196.210}+\frac{1}{3}<2,$$ which is not possible. So, $p_4=199$.

We have $\ord_3(197)=\ord_{11}(197)=2$,$\ord_5(197)=4$ and $\ord_{199}(197)=198$ are all even. Hence, none of the factors of the left hand side of equation \eqref{ord-5-5} divides $\s(197^{a_3})$, which is a contradiction.

Combining the five subcases, we conclude that $p_3<17$.

\textit{Case 2.} $p_3=13$.

Nothing again the elementary inequality $\dfrac{p}{p-1}>\dfrac{p+l}{p+l-1}$ for positive integers $p$ and $l$ we see that $D\geq 13$ cannot occur in this case, if $D\geq 11$ cannot occur when $p_4=17$. And indeed this is the case, since $$2=\frac{\s(n)}{n}+\frac{1}{D}<\frac{3.11.13.17}{2.10.11.16}+\frac{1}{11}<2,$$ is impossible. So, $D=3$ or $9$ in all these cases.

\textit{Subcase 2.1.} $D=3$.

We have the following equation in this case

\begin{equation}\label{ord-5-6}
        \s(3^{a_1}.11^{a_2}.13^{a_3}.p_4^{a_4})=5.3^{a_1-1}.11^{a_2}.13^{a_3}.p_4^{a_4}.
\end{equation}

Let us use the function $f$ defined earlier; which is this case is $$\ff=\left(1-\frac{1}{3^{a_1+1}}\right)\left(1-\frac{1}{11^{a_2+1}}\right)\left(1-\frac{1}{13^{a_3+1}}\right)\left(1-\frac{1}{p_4^{a_4+1}}\right).$$ We also introduce the following function $$\g=\frac{2^4.5^2.(p_4-1)}{3.11.13.p_4}.$$ From equation \eqref{ord-5-6}, it is clear that $$\ff=\g.$$

Clearly 
\begin{align*}
 \ff\geq \left(1-\frac{1}{3^{3}}\right)\left(1-\frac{1}{11^{3}}\right)\left(1-\frac{1}{13^{3}}\right)\left(1-\frac{1}{17^{3}}\right)&\\
 =0.961606\cdots ,
\end{align*}

\noindent and $$\g\leq \frac{2^4.5^2}{3.11.13}=0.932401\cdots ;$$ which are incompatible with each other. Hence, we get a contradiction.

\textit{Subcase 2.2.} $D=9$.

In this case, if $p_4\geq 19$, then we have $$2=\frac{\s(n)}{n}+\frac{d}{n}<\frac{3.11.13.19}{2.10.12.18}+\frac{1}{9}<2,$$ which is not possible. So, $p_4=17$. Observing that, $\ord_{17}(3)=16=\ord_{17}(11)$ and $\ord_{17}(13)=4$, we can conclude that this case cannot occur.

Combining the two subcases we conclude that $p_3\neq 13$.

Combining the two cases above, we conclude that $p_2\neq 11$.

\end{proof}

\begin{lemma}\label{p2-f}
If $n$ is an odd deficient perfect number of the form in Theorem \ref{th-1}, then $5\leq p_2\leq 7$.
\end{lemma}

\begin{proof}
Collecting Lemmas \ref{p2}, \ref{p2-1}, \ref{p2-2}, \ref{p2-3}, \ref{p2-4} and \ref{p2-5} gives us the result.
\end{proof}

\begin{proof}[Proof of Theorem \ref{th-1}]

The first part is proved in Lemma \ref{p1}, while the second part in proved in Lemma \ref{p2-f}.

\end{proof}

\section{Other Results and Open Problems}\label{open}

Numerical evidence as quoted in Section \ref{intro} encouraged us to make the following conjectures.

\begin{conj}\label{conj-1}
There is only one odd deficient perfect number with four distinct prime factors.
\end{conj}

\begin{conj}\label{conj-2}
For any positive integer $k\geq3$, there are only finitely many odd deficient perfect numbers with exactly $k$ distinct prime factors.
\end{conj}

The case $k=3$ in Conjecture \ref{conj-2} corresponds to the main result of Tang and Feng \cite{df}. Theorem \ref{th-1} gives some evidence in support of Conjecture \ref{conj-1} and the case for $k=4$ in Conjecture \ref{conj-2} by eliminating several candidates of primes. The only cases to eliminate now are $p_2=5$ or $7$.

We should note that our methods although works for providing bounds in support of Conjecture \ref{conj-2}, however a lot of work is required to give very explicit values of primes. As an example of the type of results we are referring to, we present the following theorem.

\begin{thm}\label{th-2}
If $n$ is an odd deficient perfect number with five distinct prime factors, $p_1<p_2<p_3<p_4<p_5$ such that $n=p_1^{a_1}.p_2^{a_2}.p_3^{a_3}.p_4^{a_4}.p_5^{a_5}$ with positive integers $a_i$, then $3\leq p_1\leq 5$
\end{thm}

\begin{proof}
Indeed, if this is the case, then we have $2=\dfrac{\s(n)}{n}+\dfrac{d}{n}$ where $d$ is the deficient divisor. Clearly if $p_1\geq 7$, we have $$2=\frac{\s(n)}{n}+\frac{d}{n}<\frac{7.11.13.17.19}{6.10.12.16.18}+\frac{1}{7}<2,$$ which is impossible. So, $3\leq p_1\leq 5$.
\end{proof}

\begin{rem}
A case by case analysis of $p_1=3$ and $p_1=5$ in Theorem \ref{th-2}, as we have done in Section \ref{proof} would help in finding bounds for $p_2$, as well as eliminate some of the choices. But, we do not explore this further. It is our belief that some other method must come into place to say something about these type of results. We hope to discuss the cases for $p_2=5,7$ in a subsequent paper.
\end{rem}

\subsection*{Note Added}

The case for $p=7$ is discussed by the second author \cite{def-2}, where he proves that there is only one such deficient perfect number when $7$ divides $n$.

\subsection*{Acknowledgements}
The second author is supported by the Austrian Science Foundation FWF, START grant Y463.

\end{document}